\documentclass[twoside,12pt]{article}
\usepackage{amsmath} 
\usepackage{amssymb} 
\usepackage{amsthm}  
\usepackage{amscd}
\usepackage{mathrsfs}
\usepackage{latexsym}
\usepackage{upgreek}
\usepackage{upgreek}
\usepackage{graphicx}
\usepackage{array}
\usepackage[utf8]{inputenc}
\usepackage{tikz}
\usepackage{dcolumn}
\usepackage{multirow}
\usepackage{booktabs}
\usepackage{lscape}
\usepackage{verbatim}
\usepackage{pgfplots}
\DeclareMathAlphabet{\mathpzc}{OT1}{pzc}{m}{it}

\newtheorem{theorem}{Theorem}[section]
\newtheorem{lemma}[theorem]{Lemma}

\newcommand{\abs}[1]{\lvert#1\rvert}

\begin{document}
\title{A note on Flavell's theorem associated with Frobenius groups}
\author{Liguo He\textsuperscript{1} \thanks{E-mail address:  helg-lxy@sut.edu.cn(L.He), \,  Zhugang@habu.edu.cn (G.Zhu)},
Gang Zhu\textsuperscript{2}\\{\footnotesize 1. Dept. of Math., Shenyang University of Technology, Shenyang, 110870,  PR China}\\{\footnotesize 2. College of Teacher Education, Harbin University, Harbin, 150086,  PR China}}
\date{}
\maketitle

\noindent \textbf{Abstract.} {\footnotesize{Let $G$ be a Frobenius group with the Frobenius kernel $K$. Suppose that $G$ contains a nontrival subgroup $D \subseteq K$ such that the normalizer $N_G(D) \not\subseteq K$. When $D$ is no $2$-group, Flavell proved, without using character theory, that $K$ is a subgroup of $G$. Based on this result, we further prove that $K$ is a subgroup when $D$ is a $2$-group. }}



\section{Introduction}

Let $G$ be a finite group. If $G$ acts transitively on the finite set $\Omega = \{\alpha_1, \alpha_2, \cdots, \alpha_n\},$ $n > 1$, the fixed-point subgroup $C_G(\alpha_i)$ for each $\alpha_i$ is nontrivial and any intersection $C_G(\alpha_i) \cap C_G(\alpha_j) = 1$ for the different $i,j$, then $G$ is called a Frobenius group. Set $K = (G - \cup_{i = 1}^n C_G(\alpha_i)) \cup \{1\}$, then Frobenius group $G$ obviously has the \textit{Frobenius decomposition} $G = K \cup C_G(\alpha_1)\cup C_G(\alpha_2)\cup \cdots \cup C_G(\alpha_n)$. Usually $K$ is called the Frobenius kernel and $C_G(\alpha_i)$ a Frobenius complement. The $K$ is obviously a normal subset of $G$.  Equivalently, if $G$ contains a subgroup $H$ such that the normalizer $N_G(H) = H$ and any intersection $H^x \cap H^y = 1$ for the distinct conjugates $H^x$ and $H^y$ of $H$ in $G$ (with $x, y \in G$), then $G$ is said to be a Frobenius group. The subgroup $H$ is called a Frobenius complement.

Applying the character theory of finite groups,  Frobenius proved  that $K$ is a normal subgroup of $G$ and $G = K C_G(\alpha_i)$ with $K \cap C_G(\alpha_i) = 1$ (for example, see \cite{Fxx} or \cite[Satz V.7.6]{Hu}), which is referred to as Frobenius' theorem. As the action of $G$ on $\Omega$ is transitive, it is clear that all $C_G(\alpha_i)$ are conjugate each other by elements of $G$.

Two famous applications of the character theory of finite groups are Burnside's $p^aq^b$-theorem and the above Frobenius' theorem, which are proved at the beginning of the last century. It was then a challenge to find a proof without characters for the two theorems. The character-free proof for Burnside's theorem was reached by Bender, Goldschmidt and Mathsuyama at the early 1970's (see \cite{Benx, Golx, Matx}). For Frobenius' theorem,  the partial results were attained by Corr\'adi, Hov\'ath and Flavell in \cite{C40,F11}.

Flavell in \cite{F11} proved the following result.

\textsc{Theorem A.} \textit{Let $G$ be a Frobenius group with kernel $K$, and suppose that $G$ contains a non-$2$-subgroup $1 \subsetneq D \subseteq K$ such that $N_G(D) \not\subseteq K$. Then $K$ is a subgroup.}

Using this theorem and its proof together with Frobenius' criterion for $p$-nilpotence (7.2.4, p170, \cite{Kurz}), we prove the following result.

\textsc{Theorem B.} \textit{ Let $G$ be a Frobenius group with kernel $K$, and suppose that $G$ contains a $2$-subgroup $1 \subsetneq D \subseteq K$ such that $N_G(D) \not\subseteq K$. Then $K$ is a subgroup.}

Combining the two theorems above, the next corollary is immediate.

\textsc{Corollary C.} \textit{ Let $G$ be a Frobenius group with kernel $K$, and suppose that $G$ contains a subgroup $1 \subsetneq D \subseteq K$ such that $N_G(D) \not\subseteq K$. Then $K$ is a subgroup.}

We always use $\pi$ to denote a set of prime numbers and $\pi'$ the complement for $\pi$ in the full set of prime numbers. For a finite group $G$, write $\pi(G)$ for the set of prime divisors of the order $\abs{G}$ of $G$.

Unless otherwise stated, the notation and terminology is standard, as presented in the general finite group theory textbooks.

\section{Preliminaries}

Some useful facts are collected as follows.

\begin{lemma}\label{fa1}
Let $G$ be a Frobenius group with the Frobenius kernel $K$ and a Frobenius complement $H$, then the following statements are true.
\begin{enumerate}
\item $H$ is a Hall $\pi$-subgroup of $G$ ($\mathpzc{E}_{\pi}$-property).
\item $\abs{G} = \abs{K}\abs{H}$.
\item If $1 \neq x \in K$, then $C_G(x) \subseteq K$.
\end{enumerate}
\end{lemma}
\begin{proof}
See Lemma 2.1 of \cite{F11}, for example.
\end{proof}

\begin{lemma}\label{fa2}
Let $G$ be a Frobenius group with kernel $K$ and complement $H$, and suppose that $H$ normalizes a nontrivial subgroup contained in $K$. Then $K$ is a subgroup.
\end{lemma}
\begin{proof}
See Lemma 2.2(iv) of \cite{F11}, for example.
\end{proof}

Lemma \ref{fa1}(1) and the next result show that Frobenius complements of a Frobenius group behave just like Sylow subgroups of a finite group.

\begin{lemma}\label{hacd}
Let $G$ be a Frobenius group, then
\begin{enumerate}
\item Every two Frobenius complements are conjugate in $G$($\mathpzc{C}_{\pi}$-property).
\item Any $\pi$-subgroup is contained in some Frobenius complement ($\mathpzc{D}_{\pi}$-property).
\end{enumerate}
\end{lemma}
\begin{proof}
Immediately follow from \cite[\textsc{Theorem E}]{F11}.
\end{proof}

\begin{lemma}\label{fa3}
Let $G$ be a Frobenius group with kernel $K$ and complement $H$, write $\pi = \pi(H)$. Suppose that  $L \leq G$ is neither $\pi$ nor $\pi^{'}$-group. Then there exists some $g \in G$ such that $L \cap H^g \neq 1$ and $L$ is a Frobenius subgroup of $G$ with kernel $L \cap K$ and complement $L \cap H^g$.
\end{lemma}
\begin{proof}
Immediately follow by \cite[Corollary 3.1]{F11} and  Lemma \ref{hacd}.
\end{proof}

\section{Proof of Theorem B}

Let $H$ be a Frobenius complement. By Lemma \ref{hacd}(2), we may replace $D$ with a suitable conjugate to suppose that $N_G(D) \cap H \neq 1$. Corollary 3.3 of \cite{F11} implies that $N_G(D)\cap K$ is a normal subgroup of $N_G(D)$. Thus we may choose a subgroup $M \leq G$, maximal subject to

\begin{description}
\item [(i)] $M \cap K \neq 1$ and $M \cap H \neq 1$, and
\item [(ii)] $M \cap K$ is a subgroup.
\end{description}

Let $F = M \cap K$ and $\pi = \pi(F)$. Using the proof of \textsc{Theorem A}, we may assume that $F$ is a $2$-group and $\pi = \{ 2 \}$.

\begin{lemma}\label{fav2}
\begin{description}
\item [(i)] $F$ is a Sylow $2$-subgroup of $G$.
\item [(ii)] If $P \neq 1$ is a characteristic subgroup of $F$, then $M = N_G(P)$.
\item [(iii)] $F \trianglelefteq M$ and every $2$-subgroup of $M$ is contained in $F$.
\end{description}
\end{lemma}
\begin{proof}
  Lemma \ref{fa3} yields that $M$ is a Frobenius group with kernel $F$, so $F \in Syl_2(M)$ is normal in $M$, part(iii) follows. By the maximality  of $M$, we get $M = N_G(F) = N_G(P)$, part(ii) follows. The nilpotence of $F$ further implies $F \in Syl_2(G)$, part (i) follows.
\end{proof}

The following consequence is a complement to Lemma 4.3 in \cite{F11}, where $\pi \neq \{2 \}$.
\begin{lemma}\label{fav4}
All subgroups contained in $K$ are $2$-nilpotent.
\end{lemma}
\begin{proof}
Lemma \ref{fav2} (i) shows that $F \in Syl_2(G)$. We shall proceed induction on the quotients of $\abs{F}$ by the orders of $2$-subgroups of all subgroups in $K$. Let $L \subseteq K$ be a subgroup. If $L$ is a $2'$-group, there is nothing to prove. Choose $1 \neq P \in Syl_2(L)$. Replacing $L$ by a suitable conjugate, we may suppose that $1 < Q \leq P \leq F$. Consider the quotient $\abs{F}/\abs{Q}$. If $\abs{F} = \abs{Q}$, then $N_L(Q) = N_G(Q) \cap L \leq M \cap K = Q$, clearly $N_L(Q)$ is $2$-nilpotent. If $Q < F$, then $Q < R = N_F(Q) \leq N_G(Q) \cap K = N_K(Q)$ which is a subgroup by Corollary 3.3 of \cite{F11}. As $\abs{F:R} < \abs{F: Q}$, the induction hypothesis implies that $N_K(Q)$ is $2$-nilpotent, and thus $N_L(Q)$ ($\leq N_K(Q)$) is also $2$-nilpotent. Frobenius' $p$-nilpotence criterion (7.2.4 in \cite[p170]{Kurz}) implies that $L$ is $2$-nilpotent.
\end{proof}

\begin{lemma}\label{fav5}
Let $P$ be a nontrivial subgroup of $F$. Then
\begin{description}
\item [(i)] $N_F(P)$ is a Sylow $2$-subgroup of $N_G(P)$, and
\item [(ii)] $N_G(P) = \mathscr{O}_{2'}(N_G(P))(N_G(P) \cap M)$.
\end{description}
\end{lemma}
\begin{proof}
We argue by induction on $\abs{F:P}$. If $P = F$ then $N_G(P) = M$ and the result is trivial. Suppose now that $P < F$. Then $Q = N_F(P) > P$. By induction we obtain
$$N_G(Q) = \mathscr{O}_{2'}(N_G(Q))(N_G(Q) \cap M). \quad\quad \mbox{(1)}$$
Since $P < Q$ and $C_G(P) \subseteq K$ (by Lemma \ref{fa1}(2)), we further reach
$$\mathscr{O}_{2'}(N_G(Q)) \leq C_G(Q) \leq C_G(P) \leq N_G(P) \cap K. \quad\quad \mbox{(2)}$$
Using Dedekind's lemma, it follows via (1) and (2) that
$$ N_G(P) \cap N_G(Q) = \mathscr{O}_{2'}(N_G(Q)))(N_G(P) \cap N_G(Q) \cap M). \quad\quad \mbox{(3)}$$
By Lemma \ref{fav2}(iii) we see that $N_G(P) \cap N_G(Q) \cap F$ is a Sylow $2$-subgroup of $N_G(P) \cap N_G(Q) \cap M$ and hence of $N_G(P) \cap N_G(Q)$ (by equality (3)). Since $Q = N_F(P) = N_G(P) \cap N_G(Q) \cap F$, we attain that $Q$ is a Sylow $2$-subgroup of $N_G(P) \cap N_G(Q)$. We claim that $Q$ is also a Sylow $2$-subgroup of $N_G(P)$. If otherwise, then there would be a $2$-subgroup $\tilde{Q} \in Syl_2(N_G(P))$ such that $P < Q < \tilde{Q} \leq N_G(P)$, thus $\tilde{Q} \cap N_G(Q)/Q \leq N_G(P) \cap N_G(Q)/Q$ which is a $2'$-group, this forces $ N_{\tilde{Q}}(Q) = \tilde{Q} \cap N_G(Q) = Q $. This contradiction shows $Q \in Syl_2(N_G(P))$, and yielding (i).

By Corollary 3.3 of \cite{F11} we have $N_G(P) \cap K \trianglelefteq N_G(P)$ so (i) and the Frattini argument yield
$$ N_G(P) = (N_G(P) \cap K)(N_G(P) \cap N_G(Q)).$$
Using (3) and (2) we deduce
\[ \begin{array}{lll}
N_G(P) & = &(N_G(P) \cap K)(N_G(P) \cap N_G(Q) \cap M) \\
 & \leq &(N_G(P) \cap K)(N_G(P) \cap M) \leq N_G(P),
\end{array}\]
hence $N_G(P) = (N_G(P) \cap K)(N_G(P) \cap M)$. Because $N_G(P) \cap K$ is $2$-nilpotent (by Lemma \ref{fav4}) and $Q$ is a Sylow $2$-subgroup of $N_G(P) \cap K$, it follows  $N_G(P) \cap K = \mathscr{O}_{2'}(N_G(P)\cap K)Q$. Since $\mathscr{O}_{2'}(N_G(P)\cap K) \leq \mathscr{O}_{2'}(N_G(P))$ and $Q \leq N_G(P) \cap M,$ we derive that $$N_G(P) = \mathscr{O}_{2'}(N_G(P))(N_G(P) \cap M),$$ which proves (ii).
\end{proof}

The next consequence is ``dual" to the above Lemma \ref{fav5}.
\begin{lemma}\label{fav8}
Let the nonidentity element $ f \in F $. Then
\begin{description}
\item [(i)] $C_F(f)$ is a Sylow $2$-subgroup of $C_G(f)$, and
\item [(ii)] $C_G(f) = \mathscr{O}_{2'}(C_G(f))(C_G(f) \cap M)$.
\end{description}
\end{lemma}
\begin{proof}
Let $P = \langle f \rangle$, then $C_G(f) = C_G(P) \trianglelefteq N_G(P)$. By Lemma \ref{fav5}(i) $N_F(P)$ is a Sylow $2$-subgroup of $N_G(P)$, so it is also a Sylow $2$-subgroup of $C_G(f)N_F(P)$ which is contained in $ N_G(P)$. It is seen that $$\abs{C_G(f)N_F(P): N_F(P)} = \abs{C_G(f): (C_G(f) \cap N_F(P))} = $$ $$= \abs{C_G(f): (C_G(f) \cap N_G(P) \cap F)} =  \abs{C_G(f): (C_G(f)\cap F)} =  \abs{C_G(f):C_F(f)}$$ is a $2'$-number, thus $C_F(f)$ is a Sylow $2$-subgroup of $C_G(f)$. This is proves (i).

Lemma 2.1 (iii) shows $C_G(f) \subseteq K$, and Lemma \ref{fav4} further gives that $G_G(f)$ is $2$-nilpotent and so using part (i) we attain that $C_G(f) = \mathscr{O}_{2'}(C_G(f)) C_F(f)$. As $C_F(f) = C_M(f) = C_G(f) \cap M$, we achieve $$C_G(f) = \mathscr{O}_{2'}(C_G(f)) (C_G(f) \cap M),$$ which proves (ii).
\end{proof}

\begin{lemma}\label{fav6}
Any two elements of $F$ that are conjugate in $G$ are already conjugate in $M$.
\end{lemma}
\begin{proof}
Let $f \in F$ and $g \in G$ be such that $f^g \in F$. We see $Z(F)^g \leq C_G(f^g)$ and Lemma \ref{fav8}(i) gives that $C_F(f^g)$ is a Sylow $2$-subgroup of $C_G(f^g)$. Then there is $c \in C_G(f^g)$ so that $Z(F)^{gc} \leq C_F(f^g) \leq F$. Clearly $f^{g} = f^{gc}$, and thus it suffices to show that $gc \in M$. Because $N_G(Z(F)^{gc}) = M^{gc}$ and $F$ is a normal Sylow $2$-group of $M$, we obtain $F^{gc}$ is the unique Sylow $2$-subgroup of $N_G(Z(F)^{gc})$. Using Lemma \ref{fav5}(i), $N_F(Z(F)^{gc})$  is also a Sylow $2$-subgroup of $N_G(Z(F)^{gc})$. We may extract that $F^{gc} = N_F(Z(F)^{gc}) \leq F$ so $F^{gc} = F$, then $gc \in N_G(F) = M$, as required.
\end{proof}

\begin{lemma}\label{fav7}
$H \leq M$.
\end{lemma}
\begin{proof}
Let $\{f_1, f_2, \cdots, f_{\alpha}\}$ be a set of representatives for the conjugacy classes of nontrivial $2$-elements of $G$. Using Sylow's theorem,  we may choose $f_i \in F$ for all $i$.

For each $i$ let $$ T_i = f_i \mathscr{O}_{2'}(C_G(f_i)),$$ and we see $T_i \subseteq C_G(f_i) \subseteq K$. Recall that each element of $G$ can be expressed uniquely as a commuting product of a $2$-element and a $2'$-element. For all $i, j$ and all $g \in G$, we may deduce that $$T_i \cap T_j^g \neq \phi  \mbox{\quad if and only if \quad} i = j \mbox{ and } g \in C_G(f_i). $$
Then $$ K \supsetneq  \bigcup_{g \in G}\bigcup_{i=1}^{\alpha}T_i^g = \bigcup_{i=1}^{\alpha}\bigcup_{g \in G}T_i^g,$$
Hence   $$ \abs{K} > \sum_{i=1}^{\alpha} \abs{G:C_G(f_i)}\abs{T_i}.$$

Lemma \ref{fav8} shows that $C_G(f_i) = \mathscr{O}_{2'}(C_G(f_i))C_F(f_i)$ for each $i$.  Lemma 2.1(ii) gives that $\abs{G} = \abs{K}\abs{H}$. We therefore derive that
$$ \frac{1}{\abs{H}} >  \sum_{i=1}^{\alpha} \frac{1}{\abs{C_F(f_i)}}. $$

By Lemma \ref{hacd} we allow to replace $H$ by a suitable conjugate so that $M \cap H \neq 1$, then Lemma \ref{fa3} guarantees that $M \cap H $ and $F$ are Frobenius complement and kernel, respectively. Lemma 2.1(ii) shows $\abs{M} = \abs{H \cap M}\abs{F}$. Thus we obtain the following form of the above inequality.

$$1 > \sum_{i=1}^{\alpha} \frac{\abs{H}}{\abs{C_F(f_i)}} = \frac{\abs{H : H \cap M}}{\abs{F}}\sum_{i=1}^{\alpha}\frac{\abs{H \cap M}\abs{F}}{\abs{C_F(f_i)}}.$$ By Lemma 2.1(3) we have $C_F(f_i) = C_M(f_i)$ for all $i$, whence $$1 > \frac{\abs{H : H \cap M}}{\abs{F}}\sum_{i=1}^{\alpha}\frac{\abs{M}}{\abs{C_M(f_i)}}.$$ Lemma \ref{fav6} implies that $\{f_1, f_2, \cdots, f_{\alpha}\}$ is a set of representatives for the conjugacy classes of nontrivial $2$-elements of $M$. By Lemma \ref{fav2} the set of $2$-elements of $M$ is equal to $F - \{1\}$ thus $$1 > \frac{\abs{H : H \cap M}}{\abs{F}}(\abs{F}-1) = \abs{H : H\cap M}(1 - \frac{1}{\abs{F}}).$$ Now $\abs{F} \geq 2$ so $2 > \abs{H :H \cap M}$ and hence $H \leq M$, as desired.

Finally, by Lemmas \ref{fav7} and \ref{fa3} we have $M = F \rtimes H$, and  Lemma \ref{fa2} implies that $K$ is a normal subgroup of $G$, the whole proof is complete.
\end{proof}

\bigskip

\noindent \textbf{Acknowledgments} \quad

The authors would like to thank Prof. Jiwen Zeng for his report (in July, 2018) on the generalized Frobenius group which motivates the authors to consider this problem.  This work was supported by National Natural Science Foundation of China (Grant No.12171058).

\bibliographystyle{amsalpha}

\begin{thebibliography}{A}
\bibitem{Benx} Bender, H.(1972), A group theoretic proof of Burnside's $p^aq^b$-theorem, \textit{Math. Z.}, 126: 327--338. DOI:10.1007/BF01110337
\bibitem{C40} Corr\'adi, K. and Hov\'ath, E.(1996), Steps towards an elementary proof of Frobenius' theorem, \textit{Comm. Algebra}, 24(7): 2285--2292. DOI:10.1080/00927879608825700
\bibitem{F11} Flavell, P.(2000), A note on Frobenius groups, \textit{J. Algebra},  228: 367--376. DOI:10.1006/jabr.2000.8269
\bibitem{Fxx} Frobenius, G.(1901), \"{U}ber aufl\"{o}sbare Gruppen IV, \textit{Berl. Berchte}, 1216--1230.
\bibitem{Golx} Goldschmidt, D. M.(1970), A group theoretic proof of the $p^{\alpha}q^{\beta}$-theorem for odd primes, \textit{Math. Z.}, 113: 373--375. DOI:10.1007/BF01110506
\bibitem{Hu} Huppert, B.(1967), \textit{Endliche Gruppen I}, Springer--Verlag, Berlin-Heidelberg-New York. DOI:10.1007/978-3-642-64981-3
\bibitem{Kurz} Kurzweil, H., Stellmacher, B.(2004), \textit{The theory of finite groups: an introduction}, Springer-Verlag New York Inc.. DOI:10.1007/b97433
\bibitem{Matx} Matsuyama, H.(1973), Solvability of groups of oder $p^aq^b$, \textit{Osaka J. Math.},  10: 375--378.
\end{thebibliography}

\end{document}